\documentclass[a4paper,11pt]{amsart}
\usepackage{mathrsfs,amsmath,amssymb,amsthm}

\newcommand{\myforall}{~~{\rm for\ all}~~}

\newcommand{\myforsome}{~~{\rm for\ some}~~}

\newcommand{\myif}{~~{\rm if}~~}

\newcommand{\conjugacy}[2]{{#2} \circ {#1} \circ {#2}^{-1}}
\newcommand{\commute}[3]{{#3} \circ {#2} = {#2} \circ {#1}}

\newcommand{\Per}{{\rm Per}}
\newcommand{\percon}[2]{{(#1)} \overset{{\rm per}}{\longrightarrow} {(#2)}}

\newcommand{\tfullshift}{two-sided full shift}
\newcommand{\tsubshift}{two-sided subshift}

\newcommand{\kinsetu}{\vartriangleright}
\newcommand{\Z}{\mathbf{Z}}
\newcommand{\Nonne}{\mathbf{N}}
\newcommand{\Posint}{\Z_+}
\newcommand{\beposint}{\in \Posint}

\newcommand{\diam}{{\rm diam}}
\newcommand{\mesh}{{\rm mesh}}

\newcommand{\SFT}{{\rm SFT}} 

\newcommand{\Hom}{\mathcal{H}} 
\newcommand{\End}{\mathcal{H^+}} 
\newcommand{\U}{\mathscr{U}} 

\newcommand{\kuu}{\emptyset}
\newcommand{\notkuu}{\neq \kuu}

\newcommand{\invpi}{\pi^{-1}}

\newcommand{\enuma}{\begin{enumerate}}
\newcommand{\enumz}{\end{enumerate}}

\newtheorem{thm}{Theorem}[section]
\newtheorem{lem}[thm]{Lemma}
\newtheorem{prop}[thm]{Proposition}
\newtheorem{cor}[thm]{Corollary}

\theoremstyle{definition}

\theoremstyle{remark}

\numberwithin{equation}{section}

\begin{document}

\title[CHAIN MIXING ENDOMORPHISMS ARE APPROXIMATED]{CHAIN MIXING ENDOMORPHISMS ARE\\
APPROXIMATED BY SUBSHIFTS ON\\
THE CANTOR SET}

\author{TAKASHI SHIMOMURA}

\address{Nagoya Keizai University, Uchikubo 61-1, Inuyama 484-8504, Japan}
\curraddr{}
\email{tkshimo@nagoya-ku.ac.jp}
\thanks{}



\date{December 13, 2010}

\dedicatory{}

\commby{}

\begin{abstract}
Let f be a chain mixing continuous onto mapping from the Cantor set onto itself. Let g be a homeomorphism on the Cantor set that is topologically conjugate to a subshift. Then, homeomorphisms that are topologically conjugate to g approximate f in the topology of uniform convergence if a trivial necessary condition on the periodic points holds.
In particular, if f is a chain mixing continuous onto mapping from the Cantor set onto itself with a fixed point, then homeomorphisms on the Cantor set that are topologically conjugate to a subshift approximate f in the topology of uniform convergence.
In addition, homeomorphisms on the Cantor set that are topologically conjugate to a subshift without periodic points approximate any chain mixing continuous onto mappings from the Cantor set onto itself.
In particular, let f be a homeomorphism on the Cantor set that is topologically conjugate to a full shift.
Let g be a homeomorphism on the Cantor set that is topologically conjugate to a subshift.
Then, a sequence of homeomorphisms that is topologically conjugate to g approximates f.
\end{abstract}

\maketitle

\section{Introduction}
Let $(X,d)$ be a compact metric space.
Let $f:X \to X$ be a continuous onto mapping.
In this manuscript, the pair $(X,f)$ is called a {\it topological dynamical system}.
Let $\End(X)$ be the set of all topological dynamical systems on $X$.
For any $f$ and $g$ in $\End(X)$, we define $ d(f,g):=\sup_{x \in X }d(f(x),g(x))$.
Then, $(\End(X),d)$ is a metric space of uniform convergence.
$\Hom(X)$ denotes the set of all homeomorphisms from $X$ onto itself.
In this manuscript, we mainly consider the case in which $X$ is homeomorphic to the Cantor set, denoted by $C$.
T.~Kimura \cite[Theorem 1]{Kimura} and I \cite{Shimomura} have shown that the subset of $\Hom(C)$ consisting of all expansive homeomorphisms with the pseudo-orbit tracing property is dense in $\Hom(C)$.
$\SFT(C)$ denotes the set of all $f \in \Hom(C)$ that is topologically conjugate to some \tsubshift\ of finite type.
Then, $\SFT(C)$ coincides with the set of all expansive $f \in \Hom(C)$ with the pseudo-orbit tracing property (P.~Walters \cite[Theorem 1]{Walters}).
Therefore, $\SFT(C)$ is dense in $\Hom(C)$.
A topological dynamical system $(X,f)$ is said to be {\it topologically mixing} if for any pair of non-empty open sets $U, V \subset X$, there exists a non-negative integer $N$ such that $f^n(U) \cap V \notkuu$ for all $n > N$.
In \cite{Shimomura}, it is shown that if $f \in \Hom$ is topologically mixing, then there exists a sequence $\{g_k\}_{k = 1,2,\dots}$ of topologically mixing elements of $\SFT(C)$ such that $g_k \to f$ as $k \to \infty$.
Let $f$ be a chain mixing element of $\End(C)$ and $g$, an element of $\Hom(C)$ that is topologically conjugate to a \tsubshift.
In this manuscript, we consider the condition in which homeomorphisms that are topologically conjugate to $g$ approximate $f$.
Let $(X,f)$ be a topological dynamical system and $\delta >0$.
A sequence $\{x_i\}_{i = 0,1,\dots,l}$ of elements of $X$ is a {\it $\delta$ chain} from $x_0$ to $x_l$ if $d(f(x_i),x_{i+1}) < \delta$ for all $i = 0,1,\dots,l-1$.
Then, $l$ is called the length of the chain.
A topological dynamical system $(X,f)$ is {\it chain mixing} if for every $\delta >0$ and every pair $x,y \in X$, there exists a positive integer $N$ such that for all $n > N$, there exists a $\delta$ chain from $x$ to $y$ of length $n$.
Let $(X,f)$ and $(Y,g)$ be topological dynamical systems.
We write $(Y,g) \kinsetu (X,f)$ if there exists a sequence of homeomorphisms $\{\psi_k\}_{k=1,2,\dots}$ from $Y$ onto $X$ such that
 $\conjugacy{g}{\psi_k} \to f$ as $k \to \infty$.
If $(Y,g) \kinsetu (X,f)$ and if $g^n$ has a fixed point for some positive integer $n$, then $f^n$ must also have a fixed point.
We write $\percon{Y,g}{X,f}$ if this trivial necessary condition on periodic points holds.
We show the following:
\begin{thm}\label{thm:main}
Let $X$ be homeomorphic to the Cantor set.
Let $(X,f)$ be a chain mixing topological dynamical system.
Let $(\Lambda,\sigma)$ be a \tsubshift\ such that $\Lambda$ is homeomorphic to $C$.
Then, the following conditions are equivalent:
\enuma
\item $\percon{\Lambda,\sigma}{X,f}$;
\item $(\Lambda,\sigma) \kinsetu (X,f)$. 
\enumz
\end{thm}

\begin{cor}\label{cor:main1}
Let $X$ be homeomorphic to the Cantor set.
Let $(X,f)$ be a chain mixing topological dynamical system with a fixed point.
Let $(\Lambda,\sigma)$ be a \tsubshift\ such that $\Lambda$ is homeomorphic to $C$.
Then, $(\Lambda,\sigma) \kinsetu (X,f)$.
\end{cor}

\begin{cor}\label{cor:main2}
Let $X$ be homeomorphic to the Cantor set.
Let $(X,f)$ be a chain mixing topological dynamical system.
Let $(\Lambda,\sigma)$ be a \tsubshift\ such that $\Lambda$ is homeomorphic to $C$ without periodic points.
Then, $(\Lambda,\sigma) \kinsetu (X,f)$.
\end{cor}

\begin{cor}\label{cor:fullshifts}
Let $n > 1$ be an integer.
Let $(\Sigma_n,\sigma)$ be the \tfullshift\ of $n$ symbols.
Let $(\Lambda,\sigma)$ be a \tsubshift\ such that $\Lambda$ is homeomorphic to $C$.
Then, $(\Lambda,\sigma) \kinsetu (\Sigma_n,\sigma)$.
\end{cor}

\vspace{5mm}
{\sc Acknowledgments.}
The author would like to thank Professor K.~Shiraiwa for the valuable conversations and the suggestions regarding the first version of this manuscript.

\section{Preliminaries}
Let $\Z$ denote the set of all integers; $\Nonne$, the set of all nonnegative integers; and $\Posint$, the set of all positive integers.
Let $(X,d)$ be a compact metric space. For any two maps $f$ and $g$ from $X$ to  itself, we define $d(f,g) := \sup \{d (f(x),g(x)) ~|~x \in X\}$.

Let $f : X \to X$ be a continuous onto mapping.
In this manuscript, $(X,f)$ is called a {\it topological dynamical system}.
%
%
A topological dynamical system $(X,f)$ is {\it topologically conjugate} to a topological dynamical system $(Y,g)$ if there exists a homeomorphism $\psi : Y \to X$ such that $\commute{g}{\psi}{f}$.
Such a homeomorphism is called a {\it topological conjugacy}.
In this manuscript, we write $(Y,g) \kinsetu (X,f)$ if there exists a sequence of homeomorphisms $\{\psi_k\}_{k = 1,2,\dots}$ from $Y$ onto $X$ such that $\conjugacy{g}{\psi_k} \to f$ as $k \to \infty$, i.e. $d(\conjugacy{g}{\psi_k},f) \to 0$ as $k \to \infty$.
\begin{lem}\label{lem:seqkinsetu}
Let $(X,f)$ be a topological dynamical system.
Let $(Y_k,g_k)$ $(k = 1,2,\dots)$ be a sequence of topological dynamical systems.
Suppose that there exists a sequence of homeomorphisms $\psi_k : Y_k \to X$ such that $\conjugacy{g_k}{\psi_k} \to f$ as $k \to \infty$.
Let $(Z,h)$ be a topological dynamical system such that $(Z,h) \kinsetu (Y_k,g_k)$ for all $k = 1,2,\dots$.
Then, $(Z,h) \kinsetu (X,f)$.
\end{lem}
\begin{proof}
Let $\epsilon > 0$.
Then, there exists $N \beposint$ such that $d(\conjugacy{g_k}{\psi_k}, f) < \epsilon / 2$ for all $k > N$.
Assume $k > N$.
Let $\delta > 0$ be such that if $d(y,y') < \delta$, then $d(\psi_k(y),\psi_k(y')) < \epsilon / 2$.
Because $(Z,h) \kinsetu (Y_k,g_k)$, there exists a homeomorphism $\psi' : Z \to Y_k$ such that $d(\conjugacy{h}{\psi'},g_k) < \delta$.
Therefore, we find that $d(\conjugacy{h}{(\psi_k \circ \psi')},f) < d( \psi_k \circ (\psi' \circ h \circ \psi'^{-1}) \circ \psi_k^{-1},\conjugacy{g_k}{\psi_k}) + d(\conjugacy{g_k}{\psi_k},f) < \epsilon$.
\end{proof}
For a topological dynamical system $(X,f)$, we define
\[ \Per(X,f) := \{n \beposint ~|~ f^n(x) = x \myforsome x \in X \}. \]
Let $(X,f)$ and $(Y,g)$ be topological dynamical systems.
Suppose that $(Y,g) \kinsetu (X,f)$.
Then, for each $n \beposint$, $(Y,g^n) \kinsetu (X,f^n)$.
Consider a sequence of homeomorphisms $\{\psi_k\}_{k = 1,2,\dots}$ from $Y$ onto $X$ such that $\conjugacy{g}{\psi_k} \to f$ as $k \to \infty$.
Then, for each $n \beposint$, the fixed points of $\conjugacy{g^n}{\psi_k}$ approach some of the fixed points of $f^n$.
Thus, we obtain $\Per(Y,g) \subset \Per(X,f)$.
We write $\percon{Y,g}{X,f}$ if $\Per(Y,g) \subset \Per(X,f)$.
Thus, we obtain the following:
\begin{lem}\label{lem:percon}
Let $(X,f)$ and $(Y,g)$ be topological dynamical systems.
If $(Y,g) \kinsetu (X,f)$, then $\percon{Y,g}{X,f}$.
\end{lem}

%
Let $C$ be the Cantor set in the interval $[0,1]$.
A compact metrizable totally disconnected perfect space is homeomorphic to $C$.
Therefore, any non-empty open and closed subset of $C$ is homeomorphic to $C$.
%
%
Let $V= \{v_1, v_2, \dots, v_n \}$ be a finite set of $n$ symbols with discrete topology. Let $\Sigma(V) := V^{\Z}$ with the product topology.
Then, $\Sigma(V)$ is a compact metrizable totally disconnected perfect space, and hence, it is homeomorphic to $C$. We define a homeomorphism $\sigma : \Sigma(V) \to \Sigma(V)$ as follows:
\[ \sigma(x)_{i} = x_{i+1}~\text{for all}~i \in \Z. \]
The pair $(\Sigma(V),\sigma)$ is called a {\it \tfullshift} of $n$ symbols.
%
%
%
If a closed set $\Lambda \subset \Sigma(V)$ is invariant under $\sigma$, i.e. $\sigma(\Lambda) = \Lambda$, then $(\Lambda,\sigma|_{\Lambda})$ is called a {\it \tsubshift}.
In this manuscript, $\sigma|_{\Lambda}$ is abbreviated to $\sigma$.
%
%
A directed graph $G$ is a pair $(V,E)$ of a finite set $V$ of vertices and a set of directed edges $E \subset V \times V$.
Let $G = (V,E)$ be a directed graph.
$\Sigma(G)$ denotes the \tsubshift\ defined as follows:
\[ \Sigma(G) := \{ x \in V^{\Z} ~|~ (x_i, x_{i+1}) \in E \myforall i \in \Z \}. \]
A \tsubshift\ is said to be of {\it finite type} if it is topologically conjugate to $(\Sigma(G),\sigma)$ for some directed graph $G$.
Throughout this manuscript, unless otherwise stated, we assume that all the vertices appear in some element of $\Sigma(G)$, i.e. all the vertices of $G$ have both at least one indegree and at least one outdegree.
We define a set of words of length $k$ in $\Sigma(G)$ as follows:
\[ W(k,G) := \{ w_0 w_1 \cdots w_{k-1} \in V^{\{0,1,\dots,k-1\}}~|~ (w_i,w_{i+1}) \in E \myforall i = 0,1,\dots,k-2\}. \]
For a word $w = a_0 a_1 \dots a_{k-1}$ of length $k$ and an integer $m$, we define a subset $C_m(w) \subset \Sigma(G)$ as follows:
\[ C_m(w) =  \{ x \in \Sigma(G) ~|~ x_{m+i} = a_i \myforall i = 0,1,\dots,k-1 \}. \]
Such a set is called a {\it cylinder}.
Because $C_m(w)$ is an open and closed subset of $\Sigma(G)$, if $\Sigma(G)$ is homeomorphic to $C$ and if $C_m(w)$ is not empty, then $C_m(w)$ is also homeomorphic to $C$.
A word $a_0a_1\dots a_{k-1} \in W(k,G)$ is also called a {\it path} of length $k-1$ from $a_0$ to $a_{k-1}$ in $G$.
Let $x$ be an element of some \tsubshift. Let $i \leq j$ be integers.
Then, a word $x_{i} \dots x_{j}$ is also called a {\it segment} of length $j-i+1$.
\begin{lem}[{\rm Lemma 1.3 of R.~Bowen \cite{Bowen}}]\label{lem:bowen}
Let $G = (V,E)$ be a directed graph.
Suppose that every vertex of $V$ has both at least one outdegree and at least one indegree.
Then, $\Sigma(G)$ is topologically mixing if and only if there exists an $N \beposint$ such that for any pair of vertices $u$ and $v$ of $V$, there exists a path from $u$ to $v$ of length $n > N$.
\end{lem}
\begin{proof}See Lemma 1.3 of R.~Bowen \cite{Bowen}.
\end{proof}
Let $f : X \to X$ be a mapping and $\U$, a covering of $X$.
For the sake of conciseness, we define a directed graph $G_{f,\U} = (V_{f,\U},E_{f,\U}$) as follows:
\[ V_{f,\U} = \U~\text{and} \]
\[(a_0, a_1)~\in~E_{f,\U}~\myif f(a_0) \cap a_1 \neq \kuu. \]
Note that if $\kuu \notin \U$, then all the vertices have at least one outdegree.
In addition, if $f$ is an onto mapping, then all the vertices have at least one indegree. 
Let $(X,d)$ be a compact metric space and $K \subset X$.
The diameter of $K$ is defined by $\diam (K) := \sup \{d(x,y) ~|~ x,y \in K\}$.
For a finite covering $\U$ of $X$, we define $\mesh(\U) :=  \max \{ \diam(U) ~|~ U \in \U \}$.
\begin{lem}\label{lem:cont} Let $(X,d)$ be a compact metric space and $f : X \to X$, a continuous mapping.
Then, for any $\epsilon > 0$, there exists $\delta = \delta(f,\epsilon) > 0$ such that
\[\delta < \frac{\epsilon}{2};\] 
\[\text{if}~d(x,y) \leq \delta,~\text{then}~d(f(x),f(y)) < \frac{\epsilon}{2}\ \ \text{for all}\ \ x,y \in X.\]
\end{lem}
\begin{proof}
This lemma directly follows from the uniform continuity of $f$.
\end{proof}
For two directed graphs $G=(V,E)$ and $G'=(V',E')$, $G$ is said to be a {\it subgraph} of $G'$ if $V \subseteq V'$ and $E \subseteq E'$.
\begin{lem}\label{lem:contgraph}
Let $(X,d)$ be a compact metric space, $f : X \to X$ be a continuous mapping, and $\epsilon > 0$. Let $\delta = \delta(f,\epsilon)$ be as in lemma \ref{lem:cont} and $\U$, a finite covering of X such that $\mesh(\U) < \delta$.
Let $g~:~X~\to~X$ be a mapping such that $G_{g,\U}$ is a subgraph of $G_{f,\U}$.
Then, $d(f,g) < \epsilon$.
\end{lem}
\begin{proof}
Let $x \in X$. Then, $x \in U$ and $g(x) \in U'$ for some $U,U' \in \U$. Because $G_{g,\U}$ is a subgraph of $G_{f,\U}$, there exists a $y \in U$ such that $f(y) \in U'$. Therefore, from lemma \ref{lem:cont}, it follows that
\[d(f(x),g(x)) \leq d(f(x),f(y)) + d(f(y),g(x)) < \frac{\epsilon}{2} + \diam(U') < \epsilon. \]
\end{proof}
From this lemma, we obtain the following:
\begin{lem}\label{lem:kinji}
Let $(X,d)$ be a compact metric space; $f : X \to X$, a continuous mapping; and $\{\U_k\}_{k = 1,2, \dots}$, a sequence of coverings of $X$ such that $\mesh(\U_k) \to 0$ as $k \to \infty$.
Let $\{g_k\}_{k=1,2,\dots}$ be a sequence of mappings from $X$ to $X$ such that $G_{g_k,\U_k}$ is a subgraph of $G_{f,\U_k}$ for all $k$.
Then, $g_k \to f$ as $k \to \infty$. 
\end{lem}
A covering $\U$ of $X$ is called a partition if $U \cap U' = \kuu$ for all $U,U' \in \U$, where $U \neq U'$.
The Cantor set has a partition by open and closed subsets of an arbitrarily small mesh.
\begin{lem}\label{lem:mix-Cantor}
Let $G = (V,E)$ be a directed graph.
Suppose that every vertex of $G$ has both at least one outdegree and at least one indegree.
Suppose that $\Sigma(G)$ is topologically mixing and that $\Sigma(G)$ is not a single point.
Then, $\Sigma(G)$ is homeomorphic to $C$.
\end{lem}
\begin{proof}
Suppose that $\Sigma(G)$ is topologically mixing.
Then, by lemma \ref{lem:bowen}, there exists an $N \beposint$ such that for any pair $u$ and $v$ of vertices of $G$, there exists a path from $u$ to $v$ of length $n$ for all $n > N$.
Then, it is easy to check that every point $x \in \Sigma(G)$ is not isolated.
Hence, $\Sigma(G)$ is homeomorphic to $C$.
\end{proof}
\section{Proof of the main result}
In this section, we prove certain lemmas and propositions in order to prove the main result.
For a mapping $\pi : Y \to X$ and a covering $\U$ of X, the covering $\{\invpi(U) ~|~ U \in \U\}$ is denoted by $\invpi(\U)$.
For any mapping $g : Y \to Y$, we define a directed graph $G_{g,\pi,\U} = (V,E)$ as follows:
\[ V = \U ;\]
\[ E = \{(a_0,a_1) \in \U \times \U ~|~ \pi(g(\invpi(a_0))) \cap a_1 \neq \kuu \}. \]
A vertex $a$ in $G_{g,\pi,\U}$ has at least one outdegree if $\invpi(a) \neq \kuu$.
\begin{lem}\label{lem:subgraph}
Let $X$ and $Y$ be homeomorphic to $C$.
Let $f : X \to X$ be a continuous mapping; $g: Y \to Y$, a mapping; and $\U_k$, a sequence of finite partitions of $X$ by non-empty open and closed subsets such that $\mesh(\U_k) \to 0$ as $k \to \infty$.
Suppose that there exists a sequence $\pi_k~(k=1,2,\dots)$ of continuous mappings from $Y$ to $X$ such that $\pi_k(Y) \cap U \neq \kuu$ for all $U \in \U_k$ and that the directed graph $G_{g,\pi_k,\U_k}$ is a subgraph of $G_{f,\U_k}$ for all $k \beposint$.
Then, there exists a sequence $\psi_k~(k=1,2,\dots)$ of homeomorphisms from $Y$ onto $X$ such that $\conjugacy{g}{\psi_k} \to f$ as $k \to \infty$.
\end{lem}
\begin{proof}
Let $k \beposint$ be fixed. By assumption, for each $U \in \U_k$, $\invpi_k(U)$ is a non-empty open and closed subset of $Y$.
Therefore, there exists a homeomorphism $\psi_k : Y \to X$ such that $\psi_k(\invpi_k(U)) = U$ for each $U \in \U_k$.
By construction, $G_{\conjugacy{g}{\psi_k},\U_k}~=~G_{g,\pi_k,\U_k}$.
By assumption, $G_{\conjugacy{g}{\psi_k},\U_k}$ is a subgraph of $G_{f,\U_k}$.
Therefore, the conclusion follows from lemma \ref{lem:kinji}.
\end{proof}
\begin{lem}\label{lem:factorinto}
Let $X$ and $Y$ be homeomorphic to $C$.
Let $f : X \to X$ be a continuous mapping; $g: Y \to Y$, a mapping; and $\U_k$, a sequence of finite partitions of $X$ by non-empty open and closed subsets such that $\mesh(\U_k) \to 0$ as $k \to \infty$.
Suppose that there exists a sequence of continuous mappings $\pi_k : Y \to X$ such that $\commute{f}{\pi_k}{g}$ and that $\pi_k(Y) \cap U \notkuu$ for all $U \in \U_k$.
Then, there exists a sequence $\psi_k~(k=1,2,\dots)$ of homeomorphisms from $Y$ onto $X$ such that $\conjugacy{g}{\psi_k}$ converges uniformly to $f$.
\end{lem}
\begin{proof}
Let $k \beposint$ be fixed. By assumption, for each $U \in \U_k$, $\invpi_k(U)$ is a non-empty open and closed subset of $Y$.
Therefore, there exists a homeomorphism $\psi_k : Y \to X$ such that $\psi_k(\invpi_k(U)) = U$ for each $U \in \U_k$.
Because $\pi_k(g(\invpi_k(U)))) = f(\pi_k(\invpi_k(U))) \subset f(U)$, $G_{g,\pi_k,\U_k}$ is a subgraph of $G_{f,\U_k}$.
Therefore, the conclusion follows from lemma \ref{lem:subgraph}.
\end{proof}

Let $\Lambda$ be a \tsubshift\ and $x \in \Lambda$.
Then, for $k < l$, a word $x_{k}x_{k+1}\dots x_{l}$ is said to be $j$-periodic if $k \leq i < i+j \leq l$ implies $x_{i} = x_{i+j}$.
\begin{lem}[{\rm Krieger's Marker Lemma, (2.2) of M.~Boyle} \cite{Boyle}]\label{lem:Krieger}
Let $(\Lambda,\sigma)$ be a \tsubshift. Given $k > N > 1$, there exists a closed and open set $F$ such that
\enuma 
\item the sets $\sigma^l(F), 0 \leq l < N$, are disjoint, and
\item if $x \in \Lambda$ and $x_{-k}\dots x_{k}$ is not a $j$-periodic word for any $j<N$, then
\[ x \in \bigcup_{-N < l < N} \sigma^l(F)\ . \]
\enumz
\end{lem}
\begin{proof}
See M.~Boyle \cite[(2.2)]{Boyle}.
\end{proof}

The next lemma is essentially a part of the proof of the extension lemma given in M.~Boyle \cite[(2.4) ]{Boyle}. The proof essentially follows that of the extension lemma.

\begin{lem}\label{lem:marker-mix}
Let $(\Lambda,\sigma)$ be a \tsubshift\ and $(\Sigma,\sigma)$, a mixing \tsubshift\ of finite type.
Let $W$ be a finite set of words that appear in some elements of $\Sigma$.
Suppose that $\Lambda$ is not a finite set of periodic points and that $\percon{\Lambda,\sigma}{\Sigma,\sigma}$.
Then, there exists a continuous shift-commuting mapping $\pi : \Lambda \to \Sigma$ such that there exists an element $x \in \pi(\Lambda)$ in which all words of $W$ appear as segments of $x$.
\end{lem}
\begin{proof}

$\Sigma$ is isomorphic to $\Sigma(G)$ for some directed graph $G = (V,E)$.
Therefore, without loss of generality, we assume that $\Sigma = \Sigma(G)$.
Because $(\Sigma(G),\sigma)$ is a mixing subshift of finite type, there exists an $n>0$ such that for every pair of elements $v,v' \in V$ and every $m \geq n$, there exists a word of the form $v\dots v'$ of length $m$.
In addition, there exists an element $\bar{x} \in \Sigma(G)$ such that $\bar{x}$ contains all words of $W$ as segments.
Let $w_0$ be a segment of $\bar{x}$ that contains all words of $W$.
Let $n_0$ be the length of the word $w_0$.
Let $N = 2n + n_0$.
If $v,v' \in V$ and $m \geq N$, then there exists a word of the form $v \dots w_0 \dots v'$ of length $m \geq N$.
Let $k > 2N$.
Using Krieger's marker lemma, there exists a closed and open subset $F \subset \Lambda$ such that the following conditions hold:
\enuma
\item\label{item:disjoint} the sets $\sigma^l(F), 0 \leq l < N$, are disjoint;
\item\label{item:periodic} if $x \in \Lambda$ and $x \notin \bigcup_{-N < l < N}\sigma^l(F)$, then $x_{-k}\dots x_{k}$ is a $j$-periodic word for some $j < N$;
\item the number $k$ is large enough to ensure that if $j$ is less than $N$ and a $j$-periodic word of length $2k+1$ occurs in some element of $\Lambda$, then that word defines a $j$-periodic orbit which actually occurs in $\Lambda$.
\enumz
The existence of $k$ follows from the compactness of $\Lambda$.
Let $x \in \Lambda$.
If $\sigma^i(x) \in F$, then we {\it mark} $x$ at position $i$.
There exists a large number $L>0$ such that whether or not $\sigma^i(x) \in F$ is determined only by the $2L+1$ block $x_{i-L}\dots x_{i+L}$.
If $x$ is marked at position $i$, then $x$ is unmarked for position $l$ with $i < l < i+N$.
Suppose that $x_i \dots x_{i'}$ is a segment of $x$ such that $x$ is marked at $i$ and $i'$ and that $x$ is unmarked at $l$ for all $i < l <i'$.
Then, $i' - i \geq N$.
If $x \in \bigcup_{-N < l < N} \sigma^l(F)$, then $x$ is marked at some $i$ where $-N < i < N$.
Suppose that $x_{-N+1}\dots x_{N-1}$ is an unmarked segment.
Then, $x \notin  \bigcup_{-N < l < N} \sigma^l(F)$, and according to condition (\ref{item:periodic}) $x_{-k}\dots x_{k}$ is a $j$-periodic word for some $j < N$.
Suppose that $x_{i} \dots x_{i'}$ is an unmarked segment of length at least $2N-1$, i.e. $i'-i \geq 2N-2$.
Then, for each $l$ with $i+N-1 \leq l \leq i'-N+1$, $x_{l-k}\dots x_{l+k}$ is a $j$-periodic word for some $j < N$.
Therefore, it is easy to check that $x_{i+N-1-k} \dots x_{i'-N+1+k}$ is a $j$-periodic word for some $j < N$.
In this proof, we call a maximal unmarked segment an {\it interval}. 
Let $x \in \Lambda$.
Let $\dots x_i$ be a left infinite interval.
Then, it is $j$-periodic for some $j < N$.
Similarly, a right infinite interval $x_i \dots $ is $j$-periodic for some $j < N$.
If $x$ itself is an interval, then it is a periodic point with period $j < N$.
If an interval is finite, then it has a length of at least $N-1$.
We call intervals of length less than $2N - 1$ as {\it short} intervals.
We call intervals of length greater than or equal to $2N - 1$ as {\it long } intervals.
If $x$ has a long interval $x_i \dots x_{i'}$, then $x_{i+N-1-k} \dots x_{i'-N+1+k}$ is $j$-periodic for some $j < N$.
We have to construct a shift-commuting mapping $\phi : \Lambda \to \Sigma$.
Let $V'$ be the set of symbols of $\Lambda$.
Let $\Phi : V' \to V$ be an arbitrary mapping.
Let $x \in \Lambda$. Suppose that $x$ is marked at $i$.
Then, we let $(\phi(x))_{i}$ be $\Phi(x_{i})$.
We map periodic points of period $j < N$ to periodic points of $\Sigma$.
Then, we construct a coding of $\phi(x)$ in three parts.
For any $(v,v',l) \in V\times V \times \{N-1,N,N+1,\dots,2N-2\}$, we choose a word $\Psi(v,v',l)$ in $G$ of length $l$ such that the word of the form $v \Psi(v,v',l) v'$ is a path in $G$.

(A) {\it Coding for short interval:}
Let $x_i \dots x_{i'}$ be a short interval. Then, $x$ is marked at $i-1$ and $i'+1$.
We have already defined a code for position $i-1$ and $i'+1$ as $\Phi(x_{i-1})$ and $\Phi(x_{i'+1})$, respectively.
The coding for $\{i, i+1, i+2, \dots, i' \}$ is defined by the path $\Psi(\Phi(x_{i-1}),\Phi(x_{i'+1}),i'-i+1)$.

(B) {\it Coding for periodic segment: }
For an infinite or a long interval, there exists a corresponding periodic point of $\Lambda$. The periodic points of $\Lambda$ are already mapped to periodic points of $\Sigma$.
Therefore, an infinite or a long periodic segment can be mapped to a naturally corresponding periodic segment.

(C) {\it Coding for transition part: }
To consider a transition segment, let $x_{i} \dots x_{i'}$ be a long interval.
Then, $x_{i-1}$ has already been mapped to $\Phi(x_{i-1})$ and $x_{i+N-1}$ is mapped according to periodic points. Assume $x_{i+N-1}$ is mapped to $v_0$.
The segment $x_{i-1} \dots x_{i+N-1}$ has length $N+1$.
We map the segment $x_{i} \dots x_{i+N-2}$ to $\Psi(\Phi(x_{i-1}),v_0,N-1)$.
In the same manner, the transition coding of right hand side of a long interval is defined.
In the same manner, the transition coding of the left or the right infinite interval is defined.
It is easy to check that there exists a large number $L'>0$ such that the coding of $(\phi(x))_i$ is determined only by the block $x_{i-L'} \dots x_{i+L'}$.
Therefore, $\phi : \Lambda \to \Sigma$ is continuous.
Because $\Lambda$ is not a set of finite periodic points, there exists an $x \in \Lambda$ such that $x$ contains at least one transition segment or at least one short interval.
In the above coding, we can take $\Psi$ such that a short interval or a transition segment is mapped to a word that involves $w_0$.
\end{proof}

\begin{prop}\label{prop:main}
Let $(\Sigma,\sigma)$ be a topologically mixing \tsubshift\ of finite type such that $\Sigma$ is homeomorphic to $C$.
Let $(\Lambda,\sigma)$ be a \tsubshift\ such that $\Lambda$ is homeomorphic to $C$.


Then, $(\Lambda,\sigma) \kinsetu (\Sigma,\sigma)$ if and only if $\percon{\Lambda,\sigma}{\Sigma,\sigma}$.
\end{prop}
\begin{proof}
If $(\Lambda,\sigma) \kinsetu (\Sigma,\sigma)$, then by lemma \ref{lem:percon}, we obtain $\percon{\Lambda,\sigma}{\Sigma,\sigma}$.
Suppose that $\percon{\Lambda,\sigma}{\Sigma,\sigma}$.
Without loss of generality, we can assume that $\Sigma = \Sigma(G)$ for some directed graph $G = (V,E)$.
We assume that every vertex of $V$ has both at least one outdegree and at least one indegree.
Let $k \beposint$.
Because $(\Sigma,\sigma)$ is topologically mixing, by lemma \ref{lem:marker-mix}, there exists a continuous shift-commuting mapping $\pi_k : \Lambda \to \Sigma$ and $x \in \pi_k(\Lambda)$ such that $x$ contains all words of length $2k+1$ of $\Sigma$.
Let $\U_k = \{C_{-k}(w) ~|~ w \in W(2k+1,G)\}$.
Then, $\pi_k(\Lambda) \cap U \notkuu$ for all $U \in \U_k$.
Because $k$ is arbitrary, by lemma \ref{lem:factorinto}, we conclude that $(\Lambda,\sigma) \kinsetu (\Sigma(G),\sigma)$.
\end{proof}

Proof of Theorem \ref{thm:main}
\begin{proof}
If $(\Lambda,\sigma) \kinsetu (X,f)$, then by lemma \ref{lem:percon}, we obtain  $\percon{\Lambda,\sigma}{X,f}$.
Let $\percon{\Lambda,\sigma}{X,f}$ hold.
Consider a sequence $\{\U_k\}_{k=1,2,\dots}$ of partitions of $X$ by non-empty open and closed subsets such that $\mesh(\U_k) \to 0$ as $k \to \infty$.
Assume $k \beposint$.
Let $G_k = G_{f,\U_k}$.
Let $\delta > 0$ be such that any $x,x' \in X$ with $d(x,x') < \delta$ are contained in the same element of $\U_k$.
Let $\{x_0,x_1\}$ be a $\delta$ chain.
Let $U,U' \in \U_k$ be such that $x_0 \in U$ and that $x_1 \in U'$.
Then, $f(U) \cap U' \notkuu$.
Therefore, $(U,U')$ is an edge of $G_k$.
Let $U,V \in \U_k$.
Let $x \in U$ and $y \in V$.
Because $f$ is chain mixing, there exists an $N > 0$ such that for every $n > N$, there exists a $\delta$ chain from $x$ to $y$ of length $n$.
Therefore, for every $n > N$, there exists a path in $G_k$ from $U$ to $V$ of length $n$.
From Lemma \ref{lem:bowen}, $(\Sigma(G_k),\sigma)$ is topologically mixing.
By lemma \ref{lem:mix-Cantor}, $\Sigma(G_k)$ is homeomorphic to $C$.
Therefore, there exists a homeomorphism $\psi_k : \Sigma(G_k) \to X$ such that for any vertex $u$ of $G_k$, $\psi_k(C_0(u)) = u$.
By construction, we obtain $G_{\conjugacy{\sigma}{\psi_k},\U_k} = G_{f,\U_k}$.
Because $\mesh(\U_k) \to 0$ as $k \to \infty$, by lemma \ref{lem:kinji}, we find that $\conjugacy{\sigma}{\psi_k} \to f$ as $k \to \infty$.
On the other hand, it is easy to verify that $\Per(X,f) \subset \Per(\Sigma(G_k),\sigma)$.
By assumption, we obtain $\Per(\Lambda,\sigma) \subset \Per(\Sigma(G_k),\sigma)$.
From proposition \ref{prop:main}, we obtain $(\Lambda,\sigma) \kinsetu (\Sigma(G_k),\sigma)$.
Therefore, by lemma \ref{lem:seqkinsetu}, we obtain $(\Lambda,\sigma) \kinsetu (X,f)$.
\end{proof}

Proof of Corollary \ref{cor:main1}
\begin{proof}
If a topological dynamical system $(X,f)$ has a fixed point $x_0$, then $\Per(X,f) = \Posint$.
Therefore, the proof is a direct consequence of theorem \ref{thm:main}.
\end{proof}

Proof of Corollary \ref{cor:main2}
\begin{proof}
Let $(\Lambda,\sigma)$ be a \tsubshift\ without periodic points.
Then, $\Per(\Lambda,\sigma) = \kuu$.
Therefore, from theorem \ref{thm:main}, the conclusion follows.
\end{proof}

Proof of Corollary \ref{cor:fullshifts}
\begin{proof}
A \tfullshift\ is chain mixing and has a fixed point.
Therefore, the conclusion is a direct consequence of corollary \ref{cor:main1}.
\end{proof}

\vspace{3cm}

\end{document}